\documentclass{proc-l-hijacked}
\usepackage{amssymb, amsmath, amsthm,hyperref}

\newtheorem{theorem}{Theorem}[section]
\newtheorem{lemma}[theorem]{Lemma}

\newtheorem{proposition}[theorem]{Proposition}

\theoremstyle{definition}
\newtheorem{definition}[theorem]{Definition}

\theoremstyle{remark}

\numberwithin{equation}{section}

\newcommand{\Soc}{\mathrm{Soc}}
\DeclareMathOperator{\rank}{rank}

\newcommand{\Tr}{\mathrm{Tr}}
\newcommand{\Rad}{\mathrm{Rad}}
\newcommand{\cl}{\mathrm{cl}}

\begin{document}
\title[The Shoda-completion of a Banach algebra]{ The Shoda-completion of a Banach algebra  }
\author{R. Brits \and F. Schulz}
\address{Department of Pure and Applied Mathematics, University of Johannesburg, South Africa}
\email{rbrits@uj.ac.za, francoiss@uj.ac.za}

\begin{abstract}
In stark contrast to the case of finite rank operators on a Banach space, the socle of a general, complex, semisimple, and unital Banach algebra $A$ may exhibit the `pathological' property that not all traceless elements of the socle of $A$ can be expressed as the commutator of two elements belonging to the socle. The aim of this paper is to show how one may develop an extension of $A$ which removes the aforementioned problem. A naive way of achieving this is to simply embed $A$ in the algebra of bounded linear operators on $A$, i.e. the natural embedding of $A$ into $\mathcal L(A)$. But this extension is so large that it may not preserve the socle of $A$ in the extended algebra $\mathcal L(A)$. Our proposed extension, which we shall call the Shoda-completion of $A$, is natural in the sense that it is small enough for the socle of $A$ to retain the status of socle elements in the extension.
\end{abstract}

\maketitle

\section{Introduction}

Throughout this paper $A$ will be a complex, semisimple, and unital Banach algebra with the unit denoted by $\mathbf 1$. The group of invertible elements of $A$ is denoted by $G(A)$, and  the principal component of $G(A)$ by $G_{\mathbf 1}(A)$. For an element $x$ in a Banach algebra $A$ we denote the spectrum of $x$ relative to $A$ by $\sigma_A(x)$, that is 
$$\sigma_A(x):=\{\lambda\in\mathbb C:\lambda\mathbf 1-x\notin G(A)\}.$$ 
The non-zero spectrum of $x$ relative to $A$ is defined to be $\sigma^\prime_A(x):=\sigma_A(x)\setminus\{0\}$. If the underlying algebra is clear from the context, then we shall agree to omit the subscript $A$ in the notation $\sigma_A(x)$ and $\sigma_A^\prime(x)$.
We also assume that $A$ has a non-trivial socle, i.e. $\Soc(A)\not=\{0\}$. The socle, first  introduced by J. Diuedonn\'e over 70 years ago, is a well-known object in Banach algebras and in more general algebraic structures (rings etc.). As in the papers \cite{structuresocle} and \cite{tracesocleident} we shall follow the relatively recent ``spectral'' point of view of the socle developed by Aupetit and Mouton in \cite{aupetitmoutontrace}; the associated notions of rank, trace, determinant etc. together with their properties will also be taken from \cite{aupetitmoutontrace}. Any further notation in the current paper follows \cite{aupetitmoutontrace, structuresocle, tracesocleident}, and, in particular, we shall denote the trace on the socle by $\Tr(\cdot)$. Central to our discussion will be the following:

\begin{definition}
Let $A$ be a semisimple Banach algebra with $\Soc(A)\not=\{0\}$. Then $A$ is called \emph{Shoda-complete} if every traceless element of $\Soc(A)$ can be expressed as a commutator of two elements belonging to $\Soc(A)$.
\end{definition}

Let $p$ be a projection (idempotent) of $A$ with $\mathrm{rank}\,(p) \leq 1$ (recall that rank one projections correspond to minimal projections). As in \cite{structuresocle} we denote by $J_{p}$ the two-sided ideal generated by $p$, that is, we define
$$J_{p} := \left\{ \sum_{j=1}^{n} x_{j}py_{j} : x_{j}, y_{j} \in A, n \in \mathbb{N} \right\}.$$ 

The following lemma, which appeared in \cite{tracesocleident}, will play an important role in the development of the subsequent theory: 

\begin{lemma}\label{s2.2}\cite[Lemma 3.5]{tracesocleident}
There exists a collection of two-sided ideals $\left\{J_{p} : p \in \mathcal{P} \right\}$ such that every element of  $ \Soc(A)$ can be written as a finite sum of members of the $J_{p}$. Moreover, the two-sided ideals are pairwise orthogonal, that is, if $p, q \in \mathcal{P}$ with $p \neq q$, then
$$J_{p}J_{q} = J_{q}J_{p} = \left\{0 \right\}.$$
\end{lemma}

In \cite{structuresocle} it was shown that the $J_{p}$ are all minimal two-sided ideals:

\begin{lemma}\label{s2.3}\cite[Lemma 2.2]{structuresocle}
Let $p$ be a projection of $A$ with $\mathrm{rank}\,(p) \leq 1$. Then $J_{p}$ is a minimal two-sided ideal.
\end{lemma}

Let $\mathcal{P}$ be a set of projections, as in Lemma \ref{s2.2}, generating $ \Soc(A)$. By the direct sum $\oplus_{p \in \mathcal{P}} Ap \otimes pA$ we denote the subset of the Cartesian product $\times_{p \in \mathcal{P}}Ap \otimes pA$ consisting of all cross sections which are zero except at a finite number of coordinate elements of $\mathcal{P}$, equiped with pointwise scalar multiplication and addition. By \cite[Corollary 2.5]{structuresocle} each $J_{p}$ is isomorphic as an algebra to $Ap \otimes pA$, where the isomorphism maps each elementary tensor $xp \otimes py$ to $xpy$. So by Lemma \ref{s2.2} it follows that:

\begin{theorem}\label{s2.4}\cite[Theorem 3.1]{structuresocle}
$ \Soc(A) \cong \oplus_{p \in \mathcal{P}} Ap \otimes pA$.
\end{theorem}

Another useful result appears in \cite{structuresocle}:

\begin{theorem}\label{b3.3}\cite[Theorem 3.3]{structuresocle}
Let $\left\{b, a_{1}, \ldots, a_{n}\right\}$ be a linearly independent set of rank one elements in $A$. Then there exists $y \in A$ such that $\sigma (by) \neq \left\{0 \right\}$ and $\sigma \left(a_{i}y\right) = \left\{0 \right\}$ for each $i \in \left\{1, \ldots, n \right\}$.
\end{theorem}

The papers \cite{structuresocle} and \cite{tracesocleident} also contain a number of characteristic properties of Shoda-complete Banach algebras, some of which will be useful for the current paper; we give a brief summary:

\begin{theorem}[Shoda-completeness]\cite{structuresocle, tracesocleident}\label{sshodachar} Let $A$ be a semisimple Banach algebra with $\Soc(A)\not=\{0\}$. Then $A$ is Shoda-complete if and only if any one the following conditions hold:
	\begin{itemize}
		\item[(i)]{$\Soc(A)$ is a minimal two-sided ideal of $A$.}
		\item[(ii)]{$\Soc(A)=J_p$ for some rank one projection $p$.}
		\item[(iii)]{If $f$ is a linear functional on $\Soc(A)$ such that $f(ab)=f(ba)$ for all $a,b\in\Soc(A)$, then $f=\alpha\Tr$ for some $\alpha\in\mathbb C$.}
		\item[(iv)]{$pAp\cong M_n(\mathbb C)$ for every projection $p$ with $\rank(p)=n$.}
	\end{itemize}			
\end{theorem}	

The canonical example of a Shoda-complete Banach algebra is $\mathcal L(X)$, but many other Banach algebras do not have this property. For instance, consider 
$$A:=M_n(\mathbb C)\oplus M_k(\mathbb C),\ \ (n,k\in\mathbb N)$$ where $M_n(\mathbb C)$ denotes the $n\times n$ complex matrices. Then $A$ is a Banach algebra under pointwise operations. It is not hard to find traceless elements of $A$ which cannot be expressed as the commutator of two elements of $A$; so $A$ is not Shoda-complete. But the algebra $B:=M_{n+k}(\mathbb C)$, which is Shoda-complete, contains an algebraically isomorphic copy of $A$. In fact, with a little effort, one may show that $B$ is the smallest semisimple extension of $A$ such that every traceless element of $A$ becomes a commutator. By the classical Wedderburn-Artin Theorem we therefore know how to extend a finite dimensional algebra to a Shoda-complete algebra. The aim of this paper is to show that this can be done for Banach algebras with infinite dimensional socles as well. However, it should be emphasized that it will be desirable that the Shoda-completion preserves the socle in the following sense: The image of socle elements in $A$ must be socle elements in the Shoda-completion of $A$. This immediately rules out the option of simply embedding $A$ in $\mathcal L(A)$ (see the remark after Theorem \ref{s3.5}). To address the problem of finding such a Shoda-completion, we start with a section on minimal and inessential ideals, and obtain some new results which will be used in the forthcoming construction; almost all of these statements are of independent interest as well. In Section~\ref{saextend} we algebraically extend $A$ to an algebra which is Shoda-complete. Then, in Section~\ref{snextend}, we introduce a submultiplicative norm on the algebraic extension which extends the norm on $A$. The subsequent norm completion is unfortunately not necessarily semisimple. However, because of the location of the radical, the quotient algebra turns out to be a semisimple Banach algebra extension of $A$ which is simultaneously Shoda-complete.

 \section{Minimal and Inessential Ideals}
 
 If $\mathcal{P}$  is a fixed set of projections generating $ \Soc(A)$, then Theorem~\ref{b3.3} can be used to show that for each rank one element $a \in \Soc(A)$, there exists a unique $p \in \mathcal{P}$ such that $a \in J_{p}$. In fact, more can be said if $a$ is a rank one projection: 
 
 \begin{lemma}\label{sminlem1}
 	If $q$ is a rank one projection belonging to $J_p$ then $q\in ApA$. 
 \end{lemma}
 
 \begin{proof}
 	It is well-known and easy to prove that if $p,q$ are rank one projections then $\dim(qAp)\leq 1$ and $\dim(pAq)\leq 1$ (see for instance \cite[Theorem 4.2]{puhltrace}). But $q\in J_p$ forces 
 	$\dim(qAp)=1=\dim(pAq)$. It thus follows that
 	\[q=\sum_{i=1}^nx_ipy_i\Rightarrow q=\sum_{i=1}^n(qx_ip)(py_iq)\Rightarrow q=qxpyq\in ApA,\]
 	for some $x$ and $y$ in $A$.
 \end{proof}
 
 \begin{lemma}\label{sminlem2}
 	If $q$ is a rank one projection belonging to $J_p$ then $q\in G_{\mathbf 1}(A)pG_{\mathbf 1}(A).$ 
 \end{lemma}
 
 \begin{proof}
 	If $p = \mathbf{1}$, then $A = \mathbb{C}$. But then $q = p$ and we are done. We may therefore assume that $p \notin G(A)$. From the preceding result we have that $q=xpy$ for some $x,y\in A$. By \cite[Theorem 2.2]{aupetitmoutontrace} we can find 
 	$z\in G_{\mathbf 1}(A)$ such that $zxp$ is a rank one projection. Now consider the entire function $f:\mathbb C\to Ap$ defined by \[f(\lambda)=(1-\lambda)p+\lambda zxp.\] We observe that
 	$\#\sigma(f(\lambda))\leq 2$ for all $\lambda\in\mathbb C$. Moreover, $\#\sigma(f(\lambda))=2$ is actually attained for some $\lambda\in\mathbb C$. Aupetit's Scarcity Theorem \cite[Theorem 3.4.25]{aupetit1991primer} says that the set
 	\[D=\{\lambda\in\mathbb C:\#\sigma(f(\lambda))=1\}\] is closed and discrete in $\mathbb C$. If we notice that $D=\{\lambda\in\mathbb C: \Tr\left(f(\lambda)\right)=0\}$, and we define 
 	$g:\mathbb C\setminus D\to Ap$ by \[g(\lambda)=\frac{f(\lambda)}{\Tr(f(\lambda))}\] then it follows that we can find a path of projections connecting the two projections $p$ and $zxp$. By a result of Zem\'{a}nek \cite[Theorem 3.3]{idempotba} we infer that $zxp=vpv^{-1}$ for some $v\in G_{\mathbf 1}(A)$. In exactly the same manner we can find $\tilde z\in G_{\mathbf 1}(A)$  such that $py\tilde z=wpw^{-1}$ for some $w\in G_{\mathbf 1}(A)$. Together we see that
 	$zq\tilde z\in G_{\mathbf 1}(A)pG_{\mathbf 1}(A)$. Since $G_{\mathbf 1}(A)$ is a group we can solve for $q$ to get the result.   
 \end{proof}
 
 \begin{theorem}\label{ssimorb}
 	 A rank one projection $q$ belongs to $J_p$ if and only if $q$ belongs to the similarity orbit of $p$, that is, $q \in E_{p}:= \left\{upu^{-1}: u \in G_{\mathbf 1}(A)\right\}$.   
 \end{theorem}
 
 \begin{proof}
 	As before we may exclude the case where $p =\mathbf{1}$ and assume that $p \notin G(A)$. With the preceding lemma we can write 
 	\[q=e^{x_1}\cdots e^{x_n}pe^{y_1}\cdots e^{y_k}\] where the $x_i,y_i\in A$. We now define 
 	the entire function $f:\mathbb C\to \Soc(A)$ by 
 	\[f(\lambda)=e^{\lambda x_1}\cdots e^{\lambda x_n}pe^{\lambda y_1}\cdots e^{\lambda y_k},\] and we notice that for each $\lambda\in\mathbb C,$ $f(\lambda)$ is a rank one element. As in the preceding lemma we observe	$\#\sigma(f(\lambda))\leq 2$ for all $\lambda\in\mathbb C$, and that $\#\sigma(f(\lambda))=2$ is actually attained for some $\lambda\in\mathbb C$. The Scarcity Theorem together with the argument in Lemma~\ref{sminlem2} guarantees the existence a path of projections connecting the two projections $p$ and $q$. Hence by \cite[Theorem 3.3]{idempotba} we obtain $q=upu^{-1}$ for some $u\in G_{\mathbf 1}(A)$. 
 \end{proof}
 
 \begin{theorem}
 	All minimal left ideals contained in $J_p$ are isomorphic as Banach algebras. Similarly all minimal right ideals contained in $J_p$ are isomorphic as Banach algebras.
 \end{theorem}
 
 \begin{proof}
 	The minimal left ideal $Ap$ is contained in $J_p$. Let $I$ be any minimal left ideal in $J_p$. Then $I=Aq$ for some rank one projection $q$. Obviously $q\in J_p$ and so Lemma~\ref{ssimorb} says that $q=vpv^{-1}$ for some $v\in G_{\mathbf 1}(A)$. Since $v$ is invertible we have
 	\[Aq=Avpv^{-1}=Apv^{-1}=vApv^{-1}.\] Define $T_q:Ap\to Aq$ by 
 	$T(xp)=vxpv^{-1}$. Then $T_q$ is obviously linear, $T_q$ is surjective since $Aq=vApv^{-1}$, and $T_q$ is injective since $v$ is invertible. To show that  $T_q$ is multiplicative:
 	\begin{align*}
 	T_q((xp)(yp))&=\Tr(pyp)T_q(xp)\\&=
 	\Tr(pyp)vxpv^{-1}\\&=
 	vx\Tr(pyp)pv^{-1}\\&=
 	vxpypv^{-1}\\&=
 	vxpv^{-1}vypv^{-1}\\&=
 	T_q(xp)T_q(yp).
 	\end{align*}
 	This gives the result, and a similar argument works for minimal right ideals.
 \end{proof}
 
 If $A=\mathcal L(X)$, where $X$ is a Banach space, then it is well-known that $\Soc(A)$ is a minimal two-sided ideal of $A$, and in fact that $\Soc(A)\cong X\otimes X^\prime$ where $X^\prime$ is the dual space of $X$. The next result shows that $X$ can be identified (as a Banach space) with any minimal left ideal and $X^\prime$ with any minimal right ideal: 
 
 \begin{theorem}\label{s3.5}
  Let $A=\mathcal L(X)$, where $X$ is a Banach space. If $p\in A$ is any rank one projection, then 
  $$Ap\cong X\;\,\mathrm{and}\;\,pA\cong X'\;\,\mathrm{(as \;Banach\; spaces)}.$$
 \end{theorem}
 
 \begin{proof}
 Since $p$ has rank one, there exist $0\neq f_p\in X^\prime$ and $0\neq u_p\in X$ such that $p(x)=f_p(x)u_p$ for each $x\in X$. From this it follows that for each $s\in A$ we have that $(sp)(x)=f_p(x)s\left(u_p\right)$ for $x\in X$. We are therefore naturally led to the map $\Omega: Ap\rightarrow X$ defined by $\Omega(sp)=(sp)\left(u_{p}\right)$. It is clear that $\Omega$ is linear and injective. To see that $\Omega$ is surjective, let $v\in X$ be arbitrary and take $t\in A$, where $t(x)= f_{p}(x)v$ for each $x \in X$. Then $t\left(u_{p}\right) = v$, and the surjectivity follows. Furthermore, observe that  
 $$\left\|\Omega(sp)\right\| = \left\| \left(sp\right)\left(u_{p}\right) \right\| \leq\left\|u_{p}\right\|\left\|sp\right\|.$$
 So it follows, from the Open Mapping Theorem, that $Ap\cong X$ as Banach spaces. For the isomorphism $pA\cong X^\prime$, we consider the map $\Gamma:pA\rightarrow X^\prime$ defined by $\Gamma(ps)=f_p\circ ps$. It is trivial that $\Gamma$ is linear and injective. To show that $\Gamma$ is surjective, note that $f_p\left(u_{p}\right)=1$ and let $s_f\in A$ be defined by $s_f(x) =f(x)u_{p}$ for each $x \in X$, where $f\in X^\prime$. It then follows that $\Gamma(ps_f)=f$. To prove continuity, simply note that
 $$\left\|\Gamma(ps)\right\|\leq\left\|f_p\right\|\left\|ps\right\|.$$ So the isomorphism  $p A\cong X^\prime$ follows from the Open Mapping Theorem.
 \end{proof}
 
 \textbf{Remark.} Let $X$ be any infinite Banach space, and consider $A := \mathcal L(X) \oplus \mathcal L(X)$, where the operations are all pointwise. In particular, we note that $A$ is a semisimple Banach algebra with identity element. Fix any rank one projection $p$ in $\mathcal L(X)$. Then $(p, 0)$ and $(0, p)$ are both rank one projections of $A$. Moreover, $J_{(p, 0)}$ and $J_{(0, p)}$ are mutually orthogonal minimal two-sided ideals. Consequently, $\Soc(A)$ is not Shoda-complete. Isometrically and algebraically embed $A$ in $\mathcal L(A)$ using right multiplication. Then $A(p, 0) \cong \left(X, 0\right)$ as Banach spaces by Theorem \ref{s3.5}. Hence, $A(p, 0)$ is infinite-dimensional; so the image of $(p, 0)$ is not a socle element in $\mathcal L(A)$. This shows that $\mathcal L(A)$ is not a suitable candidate for a Shoda-completion of $A$. \\
 
The paper \cite{structuresocle} contains various algebraic characterizations of Shoda-completeness. Theorem~\ref{ssimorb} can be used to give a topological characterization:
 \begin{theorem}
 	$A$ is Shoda-complete if and only if any one of the following holds:
 	\begin{itemize}
 		\item[(i)]{The collection of rank one projections is connected.}
 		\item[(ii)]{ For each $n\in\mathbb N$, $\mathcal R_n:=\{a\in\Soc(A):\rank(a)=n\}$ is connected.}
 		\item[(iii)]{$\mathcal R_1$ is connected. }
 	\end{itemize}
 \end{theorem}
 \begin{proof}
 	(i) If $A$ is Shoda complete then by Theorem~\ref{sshodachar} there exists a rank one projection $p$ such that $\Soc(A)=J_p$. So any rank one projection $q$ belongs to $J_p$, which, by Theorem~\ref{ssimorb}, means that $q=upu^{-1}$ for some $u\in G_{\mathbf 1}(A)$. The result then follows from \cite[Theorem 3.3]{idempotba}. Conversely, if the rank one projections is a connected space then, by \cite[Theorem 3.3]{idempotba}, there exists some rank one projection $p$ such that any rank one projection $q$ takes the form $q=upu^{-1}$ for some $u\in G_{\mathbf 1}(A)$. By Theorem~\ref{s2.4}, and the remark preceding it, it then follows that $\Soc(A)=J_p$. So $A$ is Shoda-complete by Theorem~\ref{sshodachar}.\\
 	(ii) Let $A$ be Shoda-complete and suppose $a,b\in\mathcal R_n$. Pick $u\in G_{\mathbf 1}(A)$ such that
 	$au$ and $bu$ are maximal finite rank elements. By \cite[Theorem 2.8]{aupetitmoutontrace} we can write 
 	$$bu=\beta_1p_1+\cdots+\beta_np_n\mbox{ and }au=\alpha_1q_1+\cdots+\alpha_nq_n,$$
 	where the $\beta_i$ are distinct nonzero complex numbers and the $p_i$ are pairwise orthogonal rank one projections. A similar statement is valid for the $\alpha_i$ and the $q_i$. Fix $1\leq i\leq n$. Since $A$ is Shoda-complete $\Soc(A)=J_p$ for some rank one projection $p$. Thus, by Theorem~\ref{ssimorb} we can find a set $\{x_1,\dots,x_k\}\subset A$ and a set $\{y_1,\dots,y_l\}\subset A$ such that
 	 $$p_i=e^{x_1}\cdots e^{x_k}pe^{-x_1}\cdots e^{-x_k}\mbox{ and }q_i=e^{y_1}\cdots e^{y_l}pe^{-y_1}\cdots e^{-y_l}.$$ Without loss of generality we may assume
 	 $k\geq l$. Define $f_i:\mathbb C\to A$ by
 	 $$f_i(\lambda)=[(1-\lambda)\beta_i+\lambda\alpha_i]\left[\prod_{j=1}^ke^{(1-\lambda)x_j+\lambda y_j}\right]p\left[\prod_{j=1}^ke^{(\lambda-1)x_j-\lambda y_j}\right]$$
 	 where $y_j=0$ if $j>l$. Then define $g:\mathbb C\to A$ by 
 	 $$g(\lambda)=\sum_{i=1}^nf_i(\lambda).$$ The function $g$ is analytic from $\mathbb C$ into $\Soc(A)$ and satisfies $g(0)=bu$, $g(1)=au$. Since the rank is subadditive we have that
 	 $\rank(g(\lambda))\leq n$ for all $\lambda\in\mathbb C$ with $\rank(bu)=\rank(au)=n$. It follows from Aupetit and Mouton's Scarcity Theorem for rank \cite[Theorem 2.4]{aupetitmoutontrace} that $\{\lambda\in\mathbb C:\rank(g(\lambda))<n\}$ is closed and discrete in $\mathbb C$. So there exists an arc $\mathcal C$ in $\mathbb C$, each of whose members has rank $n$, connecting $bu$ and $au$. Since $u\in G_{\mathbf 1}(A)$ a similar argument implies that $b$ and $bu$ (as well as $a$ and $au$) can be connected by an arc consisting of rank $n$ elements. This suffices to prove that $\mathcal R_n$ is connected. For the converse suppose that $A$ is not Shoda-complete. Then we can find at least two distinct minimal two-sided ideals generated by rank one projections, say $J_p$ and $J_q$, which are orthogonal and properly contained in $\Soc(A)$. We shall use this to show that $\mathcal R_1$ is not connected. It suffices to show that $J_p\cap\mathcal R_1$ is both open and closed in $\mathcal R_1$. To see that $J_p\cap\mathcal R_1$ is closed, let $a_n$ be a sequence in $J_p\cap\mathcal R_1$ which converges to $a\in\mathcal R_1$. For the sake of a contradiction suppose that $a \notin J_p\cap\mathcal R_1$. Then $\lim a_n=a\in J_q$, where $J_{p}J_{q} = \left\{0 \right\}$. Hence, for each $x\in A$ we have that $\lim a_nxax=(ax)^2$. But $a_nxax=0$, for each $n$, means that $(ax)^2=0$ from which the Spectral Mapping Theorem together with the semisimplicity of $A$ then imply that $a=0\notin\mathcal R_1$. Thus, $J_p\cap\mathcal R_1$ is closed in $\mathcal R_1$. Similarly, $\left(\Soc(A) - J_p\right)\cap\mathcal R_1$ is closed, and so, $J_p\cap\mathcal R_1$ is open in $\mathcal R_1$. Thus, if $A$ is not Shoda-complete then $\mathcal R_1$ is not connected.\\
 	 (iii) In view of the preceding argument this is now trivial. 
 \end{proof}

The next result is a generalization of a well-known fact in the operator case: The Riesz projections corresponding to the isolated spectral values of a compact operator $T\in\mathcal L(X)$ are finite rank operators. Theorem~\ref{sgencomp} will be used in the proof of our main result later on.

\begin{theorem}\label{sgencomp}
	Let $J$ be an inessential ideal of $A$. If $s\in J$ then the Riesz projections of $s$ corresponding to nonzero spectral values have finite rank.
\end{theorem}

\begin{proof}
Let $\sigma'(s)=\{\lambda_1,\lambda_2,\dots\}$ and set, for any fixed $i\in\mathbb N$, $p:=p(\lambda_i,s)$, the Riesz projection corresponding to $\lambda_i$ and $s$. Recall that $pAp$ is a semisimple Banach algebra with identity $p$. There exists an open neighborhood $V$ of $\mathbf 1$ in $A$ such that $pxp$ is invertible in $pAp$ for each $x\in V$. Now suppose $x\in V$ and $\#\sigma_A(px)=\infty$. Then, by Jacobson's Lemma, $\#\sigma_A(pxp)=\infty=\#\sigma_{A}'(pxp)$, and, since $p \in sA$, it follows from our hypothesis on $s$ that $\sigma_{A}'(pxp)$ is a sequence converging to $0$. But this means
$\sigma_{pAp}(pxp)$ contains a sequence converging to zero, from which it follows (since the spectrum is closed) that $pxp$ cannot be invertible in $pAp$ giving a contradiction.
So $\#\sigma_A(px)<\infty$ for all $x\in V$ and a standard application of the Scarcity Theorem then says $\#\sigma_A(px)<\infty$ for all $x\in A$. Thus $\mathrm{rank}\,(p)<\infty$.
\end{proof}

\section{Algebraic Extension}\label{saextend}

Throughout this section $\mathcal{A}$ will denote an algebra over a scalar field $K$. The first proposition is well-known and it shows how one can obtain a multiplication scheme for $\mathcal{A} \otimes \mathcal{A}$ by utilizing any linear functional $f$ on $\mathcal{A}$. This will turn the vector space $\mathcal{A} \otimes \mathcal{A}$ into an algebra. We will then consider the specific case where $\mathcal{A} =  \Soc(A)$ and $f = \mathrm{Tr}$, and use this to construct a Shoda Completion for $A$.

\begin{proposition}\label{s2.1}
Let $\mathcal{A}$ be an algebra over the field $K$ and let $f$ be a linear functional on $\mathcal{A}$. Then $\mathcal{A} \otimes \mathcal{A}$ is an algebra where multiplication of elementary tensors is defined by
$$\left(a \otimes b\right)\left(c \otimes d\right) = f(bc)a \otimes d \;\,\:\left(a, b, c, d \in \mathcal{A}\right).$$
\end{proposition}

With $\mathcal{A} =  \Soc(A)$ and $f = \mathrm{Tr}$, Proposition \ref{s2.1} now readily gives that $ \Soc(A) \otimes  \Soc(A)$ is an algebra where multiplication of elementary tensors is given by
$$\left(a \otimes b\right)\left(c \otimes d\right) = \Tr(bc) a \otimes d \;\,\:\left(a, b, c, d \in  \Soc(A)\right).$$

Let $\mathcal{P}$ be a fixed set of projections generating $\Soc(A)$. For $p, q \in \mathcal{P}$, define $J_{p, q} := Ap \otimes qA$. Moreover, let
$$ A_{J} := \left\{\sum_{i =1}^{n} x_{i}: x_{i} \in J_{p, q}\;\,\mathrm{for\;some}\:\,p, q \in \mathcal{P}\;\,\mathrm{with}\;\,p \neq q, n\in\mathbb N \right\},$$
$$ A_{J}' := \left\{\sum_{i =1}^{n} x_{i}: x_{i} \in J_{p, q}\;\,\mathrm{for\;some}\:\,p, q \in \mathcal{P}, n\in\mathbb N \right\},$$
and
$$ A_{J}'' := \left\{\sum_{i =1}^{n} x_{i}: x_{i} \in J_{p, p}\;\,\mathrm{for\;some}\:\,p \in \mathcal{P}, n\in\mathbb N\right\}.$$
It is not hard to see that $ A_{J}$ is a vector subspace but not a subalgebra of $ \Soc(A) \otimes  \Soc(A)$, since it is not closed under multiplication. However, $ A_{J}'$ is indeed a subalgebra of $ \Soc(A) \otimes  \Soc(A)$. Also observe that $A_{J}' =  A_{J}'' +  A_{J}$, and that $ A_{J}'' \cong \oplus_{p \in \mathcal{P}} Ap \otimes pA$. The following lemma will be useful later on:

\begin{lemma}\label{s2.5}
Let $u \in  A_{J}'$. Then there exist a unique element $u_{S} \in  A_{J}''$ and a unique element $u_{j} \in  A_{J}$ such that $u = u_{S}+u_{j}$.
\end{lemma}

\begin{proof}
Suppose that
\begin{equation}
u = \sum_{i=1}^{n} x_{i}p_{i} \otimes p_{i}y_{i} + \sum_{i=1}^{m} z_{i}r_{i} \otimes q_{i}w_{i}
\label{seq2.2}
\end{equation}
and
\begin{equation}
u = \sum_{i=1}^{n'} x_{i}'p_{i}' \otimes p_{i}'y_{i}' + \sum_{i=1}^{m'} z_{i}'r_{i}' \otimes q_{i}'w_{i}',
\label{seq2.3}
\end{equation}
where $p_{i}, q_{i}, r_{i}, p_{i}', q_{i}', r_{i}' \in \mathcal{P}$, $r_{i} \neq q_{i}$ and $r_{i}' \neq q_{i}'$ for all $i$. To prove uniqueness it will suffice to show that
\begin{equation}
\sum_{i=1}^{n} x_{i}p_{i} \otimes p_{i}y_{i} = \sum_{i=1}^{n'} x_{i}'p_{i}' \otimes p_{i}'y_{i}'.
\label{seq2.4}
\end{equation}
For the sake of a contradiction, suppose (\ref{seq2.4}) is false. Let $u_{S} = \sum_{i=1}^{n} x_{i}p_{i} \otimes p_{i}y_{i}$, $u_{j} =  \sum_{i=1}^{m} z_{i}r_{i} \otimes q_{i}w_{i}$, $u_{S}' = \sum_{i=1}^{n'} x_{i}'p_{i}' \otimes p_{i}'y_{i}'$ and $u_{j}' = \sum_{i=1}^{m'} z_{i}'r_{i}' \otimes q_{i}'w_{i}'$. From (\ref{seq2.2}) and (\ref{seq2.3}) it follows that
\begin{equation}
0 = \left(u_{S} - u_{S}'\right) + \left(u_{j} - u_{j}'\right).
\label{seq2.5}
\end{equation}
Without loss of generality we may assume that
\begin{equation}
u_{S} - u_{S}' = \left(\sum_{i=1}^{l_{1}} x_{1, i}s_{1} \otimes s_{1}y_{1, i}\right) + \cdots + \left(\sum_{i=1}^{l_{k}} x_{k, i}s_{k} \otimes s_{k}y_{k, i}\right),
\label{seq2.6}
\end{equation}
where $s_{1}, \ldots, s_{k} \in \mathcal{P}$ and $s_{i} \neq s_{j}$ for $i \neq j$. Moreover, by \cite[Lemma 42.3]{bonsall1973complete} we may assume that $\left\{x_{i, j}s_{i} \right\}$ and $\left\{s_{i}y_{i, j}\right\}$ are linearly independent sets for all $i \in \left\{1, \ldots, k \right\}$. Since (\ref{seq2.4}) is false, one of the terms in (\ref{seq2.6}) must be nonzero, say
\begin{equation}
\sum_{i=1}^{l_{1}} x_{1, i}s_{1} \otimes s_{1}y_{1, i} \neq 0.
\label{seq2.7}
\end{equation}
Similarly, one of the terms in (\ref{seq2.7}) must be nonzero, say $x_{1, 1}s_{1} \otimes s_{1}y_{1, 1} \neq 0$. In particular, this means that $x_{1, 1}s_{1} \neq 0$, $s_{1}y_{1, 1} \neq 0$ and that $s_{1}$ is a rank one projection of $A$. By Theorem~\ref{b3.3}, the Spectral Mapping Theorem and Jacobson's Lemma (\cite[Lemma 3.1.2]{aupetit1991primer}), there exist $a_{1}, a_{2} \in A$ such that $s_{1}a_{1}x_{1, 1}s_{1} = s_{1}$, $s_{1}y_{1, 1}a_{2}s_{1} = s_{1}$, $s_{1}a_{1}x_{1, j}s_{1} = 0$ and $s_{1}y_{1, j}a_{2}s_{1} = 0$ for all $j \in \left\{2, \ldots, l_{1} \right\}$. Let $s = a_{2}s_{1} \otimes s_{1}a_{1}$. In particular, $s \neq 0$. Moreover, from (\ref{seq2.5}) it now follows that
\begin{equation}
0 = s\left(u_{S} - u_{S}'\right)s + s\left(u_{j} - u_{j}'\right)s.
\label{seq2.8}
\end{equation}
Thus, by the orthogonality of the $J_{p}$, the definition of multiplication in $ \Soc(A) \otimes  \Soc(A)$ and our choice of $s$, (\ref{seq2.8}) reduces to $0 = a_{2}s_{1} \otimes s_{1}a_{1}$. But this is absurd since $s \neq 0$. Hence, (\ref{seq2.4}) holds which completes the proof. 
\end{proof}

Let $u, v, w \in  A_{J}$ and let $\alpha \in \mathbb{C}$. In particular, since any product of elements from $ A_{J}$ belongs to $ A_{J}'$, Lemma \ref{s2.5} implies the following:
$$\alpha \left[uv\right]_{S} = \left[\left(\alpha u\right)v\right]_{S} = \left[u\left(\alpha v\right)\right]_{S},$$
$$\left[u\left(v+w\right)\right]_{S} = \left[uv\right]_{S} + \left[uw\right]_{S},\;\:\, \left[\left(v+w\right)u\right]_{S} = \left[vu\right]_{S} + \left[wu\right]_{S},$$
and
$$\left[u(vw)\right]_{S} = \left[(uv)w\right]_{S}.$$
A similar observation is true if $J$ is used as a subscript instead of $S$ above.\\

Let $B:= A \oplus  A_{J}$. If we define addition and scalar multiplication in $B$ pointwise, then surely $B$ is a vector space over $\mathbb{C}$. However, for $B$ to be an algebra, it is crucial that multiplication in $B$ is not pointwise. In fact, general multiplication in $B$ should be definable in a natural way provided we also know how to multiply elements of $A$ with elements of $ A_{J}$. This goes back to multiplication of $a \in A$ with elementary tensors $xp \otimes qy \in  A_{J}$; we readily think of
$$a\left(xp \otimes qy\right):=axp \otimes qy$$
and
$$\left(xp \otimes qy\right)a:=xp \otimes qya.$$
We can make the following observation in this regard:

\begin{lemma}\label{s2.6}
Let $\phi:  A_{J}'' \rightarrow  \Soc(A)$ be the algebra isomorphism obtained from Theorem \textnormal{\ref{s2.4}} and the remark preceding Lemma \textnormal{\ref{s2.5}}, and let $u \in  A_{J}''$. Then
$$u\left(xp \otimes qy\right) = \phi (u)\left(xp \otimes qy\right)$$
and
$$\left(xp \otimes qy\right)u = \left(xp \otimes qy\right)\phi (u)$$
for all $x, y \in A$ and $p, q \in \mathcal{P}$.
\end{lemma}

\begin{proof}
By the distributivity of both multiplication schemes, we may assume, without loss of generality, that $u = z_{1}r \otimes rz_{2}$ for some $r \in \mathcal{P}$. If $r =0$, then equality trivially holds true. So assume that $r \neq 0$. Recall that $\phi \left(z_{1}r \otimes rz_{2}\right) = z_{1}rz_{2}$. Now, if $r \neq p$, then by the orthogonality of $J_{p}$ and $J_{r}$ it follows that $rz_{2}xp = 0$ and hence $\Tr\left(rz_{2}xp\right) = 0$. Thus,
$$u\left(xp \otimes qy\right) = \phi (u)\left(xp \otimes qy\right) = 0.$$
Similarly,
$$\left(xp \otimes qy\right)u = \left(xp \otimes qy\right)\phi (u)=0$$ 
whenever $r \neq q$. We may therefore assume that $r=p$ and $r=q$. By the minimality of $r$ and the definition of the trace we have $rz_{2}xr = \Tr\left(rz_{2}xr\right)r$ and $ryz_{1}r = \Tr\left(ryz_{1}r\right)r$. Consequently,
$$u\left(xr \otimes qy\right)  = \left(z_{1}r \otimes rz_{2}\right)\left(xr \otimes qy\right)  =  \Tr\left(rz_{2}xr\right)z_{1}r \otimes qy$$
and
\begin{eqnarray*}
\phi (u)\left(xr \otimes qy\right) & = & z_{1}rz_{2}\left(xr \otimes qy\right)  =  z_{1}rz_{2}xr \otimes qy \\
& = & \Tr\left(rz_{2}xr\right)z_{1}r \otimes qy,
\end{eqnarray*}
and so, $u\left(xr \otimes qy\right) = \phi (u)\left(xr \otimes qy\right)$. Similarly, $\left(xp \otimes ry\right)u = \left(xp \otimes ry\right)\phi (u)$, so the lemma is proved.
\end{proof}

For $a \in A$ and $u = \sum_{i=1}^{n} x_{i}p_{i} \otimes q_{i}y_{i} \in  A_{J}$ we now define
$$au := \sum_{i=1}^{n} ax_{i}p_{i} \otimes q_{i}y_{i}$$
and
$$ua := \sum_{i=1}^{n} x_{i}p_{i} \otimes q_{i}y_{i}a.$$
Moreover, if $u = \sum_{i=1}^{n} x_{i}p_{i} \otimes q_{i}y_{i} \in  A_{J}$ and $u = u' = \sum_{i=1}^{n'} x_{i}'p_{i}' \otimes q_{i}'y_{i}' \in  A_{J}$, then
$$0 = a\left(u - u'\right) = \sum_{i=1}^{n} ax_{i}p_{i} \otimes q_{i}y_{i} - \sum_{i=1}^{n'} ax_{i}'p_{i}' \otimes q_{i}'y_{i}' = au-au',$$
and, similarly, $0 = ua - u'a$. This shows that this operation is well-defined. This leads us to the following operation for multiplication in $B$:
\begin{equation}
(x, u)(y, v) = \left(xy + \phi \left(\left[uv\right]_{S}\right), uy+xv+\left[uv\right]_{j}\right),
\label{seq2.1}
\end{equation}
where $x, y \in A$, $u, v \in  A_{J}$ and $\phi:  A_{J}'' \rightarrow  \Soc(A)$ is the isomorphism from Lemma \ref{s2.6} above. By Lemma \ref{s2.5}, and the observation above, the operation in (\ref{seq2.1}) is well-defined. Moreover, by Lemma \ref{s2.5} and the remark following it, associativity, distributivity and associativity with scalars are all satisfied by the operation in (\ref{seq2.1}). So $B$ with the multiplication described in (\ref{seq2.1}) is an algebra over $\mathbb{C}$. If $\mathbf{1} \in A$ is the identity element of $A$, then $\left(\mathbf{1}, 0\right)$ is the identity element of $B$. An important subalgebra of $B$ is described in the next proposition:

\begin{proposition}\label{s2.7}
$ \Soc(A) \oplus  A_{J}$ is a subalgebra of $B$. Moreover, $ A_{J}' \cong   \Soc(A) \oplus  A_{J}$.
\end{proposition}

\begin{proof}
It is routine to check that $  \Soc(A) \oplus  A_{J}$ is closed under addition, scalar multiplication and multiplication. Define $\psi :  A_{J}' \rightarrow   \Soc(A) \oplus  A_{J}$ by
$$\psi (u) = \left(\phi\left(u_{S}\right), u_{j}\right) \;\,\:\left(u \in  A_{J}'\right),$$
where $\phi:  A_{J}'' \rightarrow  \Soc(A)$ is the algebra isomorphism from Lemma \ref{s2.6} and $u = u_{S}+u_{j}$ as defined in Lemma \ref{s2.5}. We claim that $\psi$ is an algebra isomorphism: Let $u, v \in  A_{J}'$ and let $\alpha \in \mathbb{C}$. Then,
\begin{eqnarray*}
\psi \left(\alpha u\right) & = & \left(\phi\left([\alpha u]_{S}\right), \left[\alpha u\right]_{j}\right)  =  \left(\alpha \phi\left(u_{S}\right), \alpha u_{j}\right) \\
& = & \alpha \left(\phi \left(u_{S}\right), u_{j}\right) = \alpha \psi (u),
\end{eqnarray*}
\begin{eqnarray*}
\psi \left(u + v\right) & = & \left(\phi\left([u+v]_{S}\right), \left[u + v\right]_{j}\right)  =  \left(\phi\left(u_{S}+ v_{S}\right), u_{j} + v_{j}\right) \\
& = & \left(\phi \left(u_{S}\right), u_{j}\right) + \left(\phi \left(v_{S}\right), v_{j}\right) = \psi (u)+ \psi (v),
\end{eqnarray*}
and
\begin{eqnarray*}
\psi (u) \psi (v) & = & \left(\phi\left(u_{S}\right), u_{j}\right)\left(\phi\left(v_{S}\right), v_{j}\right) \\
& = & \left(\phi\left(u_{S}\right)\phi \left(v_{S}\right) + \phi \left(\left[u_{j}v_{j}\right]_{S}\right),\phi\left(u_{S}\right)v_{j} +u_{j}\phi \left(v_{S}\right) + \left[u_{j}v_{j}\right]_{j}\right) \\
& = & \left(\phi\left(u_{S}v_{S} + \left[u_{j}v_{j}\right]_{S}\right),u_{S}v_{j} +u_{j}v_{S} + \left[u_{j}v_{j}\right]_{j}\right) \\
& = & \left(\phi\left([uv]_{S}\right), [uv]_{j}\right) = \psi (uv),
\end{eqnarray*}
where we have used Lemma \ref{s2.6} and the fact that
$$uv = \left(u_{S}+u_{j}\right)\left(v_{S}+v_{j}\right) = u_{S}v_{S} + u_{j}v_{S} + u_{S}v_{j} + u_{j}v_{j},$$
and so, since $u_{j}v_{S}$ and $u_{S}v_{j}$ must be in $ A_{J}$, it follows that $[uv]_{S} = u_{S}v_{S} + \left[u_{j}v_{j}\right]_{S}$ and $[uv]_{j} = u_{j}v_{S} + u_{S}v_{j} + \left[u_{j}v_{j}\right]_{j}$. This shows that $\psi$ is an algebra homomorphism. To complete the proof we must therefore show that $\psi$ is bijective: Let $(x, u)$ be any element in $ \Soc(A) \oplus  A_{J}$. Then $\psi \left(\phi^{-1}\left(x\right) + u\right) = \left(x, u\right)$ by Lemma \ref{s2.5}. This shows that $\psi$ is surjective. To see that $\psi$ is injective, we assume that $\psi (v) = 0$ and prove that $v =0$. Now, $\psi (v) = 0$ implies that $\left( \phi\left(v_{S}\right), v_{j}\right) = 0$. Hence, $\phi\left(v_{S}\right) = 0$ and $v_{j} = 0$. But $\phi$ is an algebra isomorphism, so $v_{S} = 0$. Hence, $v = v_{S} + v_{j} = 0$. This completes the proof.
\end{proof}  

\section{Norm Extension}\label{snextend}

With the basic algebraic structure of the Shoda-completion defined for some fixed projection representative class $\mathcal{P}$ generating $\Soc(A)$, we are now in a position to extend the algebra norm on $A$ to an algebra norm on $B=A\oplus A_J$. Firstly, however, we need the following lemma:

\begin{lemma}\label{s4.1}
	Let $0 \neq u\in A_J$. Then $u$ can be uniquely expressed as $u=\sum_{i=1}^{n} u_{p_i,q_i}$,
	where $0\not=u_{p_i,q_i}\in J_{p_i,q_i}$, $p_i,q_i\in\mathcal P$ and $p_i \neq q_i$ for each $i \in \left\{1, \ldots, n \right\}$, and where $\left(p_{i}, q_{i}\right) \neq \left(p_{j}, q_{j}\right)$ for $i \neq j$ (as ordered sets).
\end{lemma}

\begin{proof}
Suppose that $u = \sum_{i=1}^{n} u_{p_{i}, q_{i}}$ as described above and that $u = \sum_{i=1}^{m} v_{r_{i}, s_{i}}$, where $0\not=v_{r_i,s_i}\in J_{r_i,s_i}$, $r_i,s_i\in\mathcal P$ and $r_i \neq s_i$ for each $i \in \left\{1, \ldots, m \right\}$, and where $\left(r_{i}, s_{i}\right) \neq \left(r_{j}, s_{j}\right)$ for $i \neq j$. Then,
\begin{equation}
0 = \sum_{i=1}^{n} u_{p_{i}, q_{i}} - \sum_{i=1}^{m} v_{r_{i}, s_{i}}.
\label{seq4.1}
\end{equation}
For the sake of a contradiction, suppose that there exists an $i \in \left\{1, \ldots, n \right\}$ such that $\left(p_{i}, q_{i}\right) \neq \left(r_{j}, s_{j}\right)$ for all $j \in \left\{1, \ldots, m \right\}$. Recall that $u_{p_{i}, q_{i}} \neq 0$ can be written as
$$u_{p_{i}, q_{i}} = \sum_{j=1}^{k} x_{j}p_{i} \otimes q_{i}y_{j},$$
where $\left\{x_{j}p_{i} \right\}$ and $\left\{q_{i}y_{j} \right\}$ are linearly independent sets. By Theorem~\ref{b3.3}, the Spectral Mapping Theorem, Jacobson's Lemma and the definition of multiplication in $\Soc(A) \otimes \Soc(A)$, there exist $x, y \in A$ such that
$$\left(p_{i} \otimes p_{i}x\right)u_{p_{i}, q_{i}}\left(yq_{i} \otimes q_{i}\right) = p_{i} \otimes q_{i} \neq 0.$$
However, from (\ref{seq4.1}) it then follows that
\begin{eqnarray*}
0 & = & \left(p_{i} \otimes p_{i}x\right)\left(\sum_{i=1}^{n} u_{p_{i}, q_{i}} - \sum_{i=1}^{m} v_{r_{i}, s_{i}}\right)\left(yq_{i} \otimes q_{i}\right) \\
& = & p_{i} \otimes q_{i},
\end{eqnarray*}
producing a contradiction. Hence, for every $i \in \left\{1, \ldots, n \right\}$ there exists a $j \in \left\{1, \ldots, m \right\}$ such that $\left(p_{i}, q_{i}\right) = \left(r_{j}, s_{j}\right)$. Similarly, it can be shown that for every $j \in \left\{1, \ldots, m \right\}$ there exists an $i \in \left\{1, \ldots, n \right\}$ such that $\left(r_{j}, s_{j}\right) = \left(p_{i}, q_{i}\right)$. Hence, (\ref{seq4.1}) can be written as
$$ 0 = \sum_{i=1}^{n} \left(u_{p_{i}, q_{i}} - v_{p_{i}, q_{i}}\right).$$
So the lemma is true if $u_{p_{i}, q_{i}} = v_{p_{i}, q_{i}}$ for each $i \in \left\{1, \ldots, n \right\}$. But if $u_{p_{i}, q_{i}} \neq v_{p_{i}, q_{i}}$ for some $i \in \left\{1, \ldots, n \right\}$, then we can obtain a contradiction using the same argument as before. We therefore have the result. 
\end{proof}

Lemma~\ref{s4.1} now allows us to write
\begin{equation*}
A_J=\bigoplus_{\substack{p,q\in\mathcal P \\
		p\not=q}}
 Ap\otimes qA.
\end{equation*}
This point of view, together with Lemma \ref{s2.5}, is crucial for the norm to be well-defined.\\

For $p\not=q$ denote by $Ap\otimes_{\pi}qA$ the algebraic tensor product endowed with the projective tensor norm, and by $Ap\,\hat\otimes_{\pi}\,qA$ its norm completion. The norm on each $Ap\,\hat\otimes_{\pi}\,qA$ will be denoted by $\|\cdot\|_{\pi,p,q}$. We denote then the algebraic direct sum of the normed algebras  $Ap\otimes_{\pi}qA$ under the $l_1$ norm by
\begin{equation*}
A_{J,\pi}=\bigoplus_{\substack{p,q\in\mathcal P \\
		p\not=q}}
Ap\otimes_\pi qA,
\end{equation*}
 and then by $\bar A_{J,\pi}$ the norm completion of $A_{J,\pi}$. It is worthwhile to notice that
 $$\bigoplus_{\substack{p,q\in\mathcal P \\ p\not=q}} Ap\,\hat\otimes_\pi\, qA\subseteq \bar A_{J,\pi},$$ but that the containment may be strict.
  With the norm on $A$ denoted by $\|\cdot\|_A$ and the $l_1$ norm on $\bar A_{J,\pi}$ by $\|\cdot\|_1$,
we have:
\begin{lemma}\label{snorm}
	$A\oplus A_{J,\pi}  $ is a normed algebra under 
	$$\|(x,u)\|=\|x\|_A+\|u\|_1,$$ where $x\in A$ and $u\in A_{J,\pi}$. 
\end{lemma}

\begin{proof}
It has been established that  $A\oplus A_{J,\pi}$ is a complex unital algebra. Since vector addition is pointwise, $\|\cdot\|$ is a vector norm. But vector multiplication is not pointwise so we need to verify that $\|\cdot\|$ is submultiplicative: Let $a,b\in A$ and $u,v\in A_{J,\pi}$.  We show that
$$\|(a,u)(b,v)\|\leq\|(a,u)\|\|(b,v)\|.$$
We can write 
$$u=\sum_{i=1}^nu_{p_i,q_i}\mbox{ and }v=\sum_{i=1}^mv_{\tilde p_i,\tilde q_i}$$ where, for each $i$,
$$u_{p_i,q_i}\in Ap_i\otimes q_iA\mbox{ and }v_{\tilde p_i,\tilde q_i}\in A\tilde p_i\otimes \tilde q_iA.$$
As was shown earlier, the product $(a,u)(b,v)$ takes the form 
$$(a,u)(b,v)=(ab+\Omega(uv),ub+av+\Gamma(uv))$$ where $\Omega(uv)$ is a sum of elements belonging to $A$, and 
$\Gamma(uv)$ is a sum of elements belonging to $A_{J,\pi}$.
We have 
\begin{align*}
\|(a,u)(b,v)\|&=\|ab+\Omega(uv)\|_A+\|ub+av+\Gamma(uv)\|_1\\
&\leq\|ab\|_A+\|\Omega(uv)\|_A+\|ub\|_1+\|av\|_1+\|\Gamma(uv)\|_1\\
&\leq\|a\|_A\|b\|_A+\|\Omega(uv)\|_A+\|ub\|_1+\|av\|_1+\|\Gamma(uv)\|_1.
\end{align*}
On the other hand,
$$\|(a,u)\|\|(b,v)\|=\|a\|_A\|b\|_A+\|u\|_1\|b\|_A+\|a\|_A\|v\|_1+\|u\|_1\|v\|_1,$$
from which we now observe that $\|\cdot\|$ will be submultiplicative provided we can establish:
\begin{itemize}
\item[(i)]{$\|ub\|_1\leq\|u\|_1\|b\|_A$.}
\item[(ii)]{$\|av\|_1\leq\|v\|_1\|a\|_A$.}
\item[(iii)]{$\|\Omega(uv)\|_A+\|\Gamma(uv)\|_1\leq\|u\|_1\|v\|_1$.}
\end{itemize}
We proceed as follows:\\
(i) We can write
$$u=\sum_{i=1}^nu_{p_i,q_i}=\sum_{i=1}^n\sum_{j=1}^{k(i)}x_{i,j}p_i\otimes q_iy_{i,j},$$
where $k(i)\in\mathbb N$ for $i\in\{1,\dots,n\}$ and $x_{i,j},y_{i,j}\in A$. Multiplication dictates that
$$ub=\sum_{i=1}^n\sum_{j=1}^{k(i)}x_{i,j}p_i\otimes q_iy_{i,j}b.$$
Notice that, for each $i$, $\sum_{j=1}^{k(i)}x_{i,j}p_i\otimes q_iy_{i,j}b$ is a representation of the tensor $u_{p_i,q_i}b\in Ap_i\otimes q_iA.$ So it follows that  
$$\|u_{p_i,q_i}b\|_{\pi,p_i,q_i}\leq
\sum_{j=1}^{k(i)}\|x_{i,j}p_i\|_A\|q_iy_{i,j}b\|_A\leq
\|b\|_A\sum_{j=1}^{k(i)}\|x_{i,j}p_i\|_A\|q_iy_{i,j}\|_A.$$
But the preceding inequality is valid for any representation 
$$\sum_{j=1}^{k(i)}x_{i,j}p_i\otimes q_iy_{i,j}\mbox{ of }u_{p_i,q_i}\in Ap_i\otimes q_iA.$$
So we infer, by the definition of the projective tensor product, that
$$\|u_{p_i,q_i}b\|_{\pi,p_i,q_i}\leq\|b\|_A\|u_{p_i,q_i}\|_{\pi,p_i,q_i}.$$
This being true for each $i$ we obtain 
$$\|ub\|_1\leq\|u\|_1\|b\|_A.$$\\
(ii) Similar to (i).\\
(iii) With $u=\sum_{i=1}^nu_{p_i,q_i}$ and $v=\sum_{i=1}^mv_{\tilde p_i,\tilde q_i}$ we observe that
$$(0,u)(0,v)=(\Omega(uv),\Gamma(uv)).$$ Furthermore:
\begin{itemize}
\item[(a)]{$u_{p_i,q_i}v_{\tilde p_j,\tilde q_j}=0$ whenever $q_i\not=\tilde p_j$.}
\item[(b)]{$\Omega\left(u_{p_i,q_i}v_{\tilde p_j,\tilde q_j}\right)\in A$ whenever $q_i=\tilde p_j$ and $p_i=\tilde q_j$.}
\item[(c)]{$u_{p_i,q_i}v_{\tilde p_j,\tilde q_j}\in A_{J,\pi}$ whenever $q_i=\tilde p_j$ and $p_i\not=\tilde q_j$.}
\end{itemize}
Consider
$$\|u\|_1\|v\|_1=\sum_{i,j}\|u_{p_i,q_i}\|_{\pi,p_i,q_i}\|v_{\tilde p_j,\tilde q_j}\|_{\pi,\tilde p_j, \tilde q_j}.$$ Now we take any term in the sum on the right for which $q_i=\tilde p_j$. If the situation in (b) prevails then $p_i=\tilde q_j$, and with $u_{p_i,q_i}=\sum_{k=1}^nx_kp_i\otimes q_iy_k$ and 
$v_{\tilde p_j,\tilde q_j}=\sum_{l=1}^m\tilde y_lq_i\otimes p_i\tilde x_l$, it then follows that 
\begin{align*}
\|u_{p_i,q_i}v_{\tilde p_j,\tilde q_j}\|&:=
\|\Omega\left(u_{p_i,q_i}v_{\tilde p_j,\tilde q_j}\right)\|_A\\&=
\left\|\sum_{k,l}\Tr(q_iy_k\tilde y_lq_i)x_kp_i\tilde x_l\right\|_A\\&\leq
\sum_{k,l}\left|\Tr(q_iy_k\tilde y_lq_i)\right|\|x_kp_i\|_A\|p_i\tilde x_l\|_A\\&=
\sum_{k,l}\left|\rho(q_iy_k\tilde y_lq_i)\right|\|x_kp_i\|_A\|p_i\tilde x_l\|_A\\&\leq
\sum_{k,l}\|q_iy_k\|_A\|\tilde y_lq_i\|_A\|x_kp_i\|_A\|p_i\tilde x_l\|_A\\&=
\left(\sum_{k=1}^n\|x_kp_i\|_A\|q_iy_k\|_A\right)\left(\sum_{l=1}^m\|\tilde y_lq_i\|_A\|p_i\tilde x_l\|_A\right).
\end{align*}
By definition of the projective tensor product we conclude that
$$\|u_{p_i,q_i}v_{\tilde p_j,\tilde q_j}\|\leq\|u_{p_i,q_i}\|_{\pi,p_i,q_i}\|\|v_{\tilde p_j,\tilde q_j}\|_{\pi,\tilde p_j,\tilde q_j}.$$
If the situation in (c) occurs, then $p_i\not=\tilde q_j$, and we have 
\begin{align*}
\|u_{p_i,q_i}v_{\tilde p_j,\tilde q_j}\|&:=
\|u_{p_i,q_i}v_{\tilde p_j,\tilde q_j}\|_1\\&=
\left\|\sum_{k,l}\Tr(q_iy_k\tilde y_lq_i)x_kp_i\otimes\tilde q_j\tilde x_l\right\|_{\pi,p_i,\tilde q_j}\\&\leq
\sum_{k,l}\|q_iy_k\|_A\|\tilde y_lq_i\|_A\|x_kp_i\|_A\|\tilde q_j\tilde x_l\|_A\\&=
\left(\sum_{k=1}^n\|x_kp_i\|_A\|q_iy_k\|_A\right)\left(\sum_{l=1}^m\|\tilde y_l\tilde p_j\|_A\|\tilde q_j\tilde x_l\|_A\right).
\end{align*}
Again the definition of the projective tensor norms on $Ap_i\otimes_\pi q_iA$ and $A\tilde p_j\otimes_\pi\tilde q_jA$ implies that
 $$\|u_{p_i,q_i}v_{\tilde p_j,\tilde q_j}\|\leq\|u_{p_i,q_i}\|_{\pi,p_i,q_i}\|\|v_{\tilde p_j,\tilde q_j}\|_{\pi,\tilde p_j,\tilde q_j}.$$
 With the partial multiplicative inequalities established and the triangle inequality we now have:
 \begin{align*}
 \|\Omega(uv)\|_A+\|\Gamma(uv)\|_1&\leq
 \sum_{i,j}\|u_{p_i,q_i}v_{\tilde p_j,q_j}\|\\&\leq
  \sum_{i,j}\|u_{p_i,q_i}\|_{\pi,p_i,q_i}\|v_{\tilde p_j,\tilde q_j}\|_{\pi,\tilde p_j, \tilde q_j}\\&=
  \|u\|_1\|v\|_1,
 \end{align*}
 which proves (iii).
\end{proof}

It follows from Lemma~\ref{snorm} that the norm completion of  $A\oplus A_{J,\pi}$ is a Banach algebra. If we denote the completion by $\tilde A_S$ then of course
$$\tilde A_S=A\oplus\bar A_{J,\pi}.$$

It is unlikely that the unital Banach algebra $\tilde A_S$ is semisimple. So we factor out the radical to obtain the semisimple
$$A_S:=\tilde A_S\big/ \Rad(\tilde A_S).$$ 
We next want to show that $A$ is isometrically embeddable into $A_S$.
This will be easy once we have:
\begin{lemma}\label{sradical}
$$ \Rad(\tilde A_S)\subseteq\{(0,u):u\in \bar A_{J,\pi}\}.$$	
\end{lemma}
\begin{proof}
Suppose $(a,u)\in \Rad(\tilde A_S)$. Then we can write $u=\lim_n u_n$ where, for each $n$, $u_n\in A_{J,\pi}$. Let $x\in A$ be arbitrary and observe that
$$(a,u)(x,0)=(ax,ux)\in \Rad(\tilde A_S).$$ So if $p$ is a rank one projection belonging to the class $\mathcal P$, and $z\in A$ is arbitrary then, since $ux=\lim_nu_nx$, we deduce that
$$(pz,0)(ax,ux)(p,0)=(pzaxp,0)\in \Rad(\tilde A_S).$$ It follows from the semisimplicity of $A$ that 
$$pzaxp=0\mbox{ for all }x,z\in A.$$ Thus by Jacobson's Lemma $\sigma_A(zaxp)=\{0\}$ for all $x,z\in A$ and, since $A$ is semisimple, we have that $axp=0$ for all $x\in A$. Therefore, since $p\in\mathcal P$ was arbitrary, we have that $axux=0$ for each $x\in A$, and hence that
$$(ax,0)(ax,ux)=((ax)^2,0)\in \Rad(\tilde A_S)$$ 
for each $x\in A$. But this means that 
$\sigma_A(ax)=\sigma_{\tilde A_S}((ax,0))=\{0\}$ holds for each $x\in A$ from which the semisimplicity of $A$ yields $a=0$.
\end{proof}	
\begin{theorem}
$A$ is isometrically embeddable into $A_S$.
\end{theorem}	
\begin{proof}
	Let $T:A\to A_S$ be the canonical map given by 
	$$a\mapsto(a,0)+ \Rad(\tilde A_S).$$ It is obvious that $T$ is a homomorphism into $A_S$. It remains to prove that $T$ is an isometry. Using Lemma~\ref{sradical} we have
	\begin{align*}
	\|Ta\|&=\|(a,0)+ \Rad(\tilde A_S)\|\\&=
	\inf\{\|(a,0)+(0,u)\|:(0,u)\in \Rad(\tilde A_S)\}\\&=
	\inf\{\|a\|_A+\|u\|_1:(0,u)\in \Rad(\tilde A_S)\}\\&=
	\|a\|_A+\inf\{\|u\|_1:(0,u)\in \Rad(\tilde A_S)\}\\&=
	\|a\|_A,
	\end{align*}
	which gives the result.
\end{proof}
The Banach algebra $A_S$ is thus a semisimple extension of $A$. We want to show that $A_S$ satisfies Shoda's Theorem. Let $p$ be a rank one projection of $ \Soc(A)$, and consider the element $(p,0)+ \Rad(\tilde A_S)\in A_S$. Let $(x,u)+ \Rad(\tilde A_S)$ be arbitrary in $A_S$. Then,
\begin{align*}
&\left[(p,0)+ \Rad(\tilde A_S)][(x,u)+ \Rad(\tilde A_S)\right]\left[(p,0)+ \Rad(\tilde A_S)\right]\\&=(p,0)(x,u)(p,0)+ \Rad(\tilde A_S)\\&=(pxp,0)+ \Rad(\tilde A_S)\\&=(\lambda p,0)+ \Rad(\tilde A_S)\\&=\lambda\left[(p,0)+ \Rad(\tilde A_S)\right],
\end{align*}
which proves that $(p,0)+ \Rad(\tilde A_S)\in\Soc\,A_S$.\\
For any pair $p,q\in\mathcal P$ where $p\not=q$, we now consider the element  $(0,p\otimes q)+ \Rad(\tilde A_S)$. Since
$$\left[(p,0)+ \Rad(\tilde A_S)\right]\left[(0,p\otimes q)+ \Rad(\tilde A_S)\right] =(0,p\otimes q)+ \Rad(\tilde A_S),$$
it follows that $(0,p\otimes q)+ \Rad(\tilde A_S)$ is a rank one element of $A_S$. Moreover, since
\begin{align*}
&\left[(x,0)+ \Rad(\tilde A_S)\right]\left[(0,p\otimes q)+ \Rad(\tilde A_S)\right]\left[(y,0)+ \Rad(\tilde A_S)\right]\\
& =  (0,xp\otimes qy)+ \Rad(\tilde A_S),
\end{align*}
for all $x, y \in A$, it follows that $(0,xp\otimes qy)+ \Rad(\tilde A_S)$ has rank less than or equal to one for all $x, y \in A$. Together we have proved the following:
\begin{proposition}
	$\Soc(A)\oplus A_{J,\pi}+ \Rad(\tilde A_S)$ is a vector subspace, but not necessarily an ideal of $\Soc\,A_S$.
\end{proposition}

So, to simplify, we shall from now on write $(a,u)$ for elements belonging to $A_S$, with the understanding that we actually mean $(a,u)+ \Rad(\tilde A_S)$.  

\begin{theorem}
The semisimple Banach algebra $A_S$ is Shoda-complete.	
\end{theorem} 
\begin{proof}
 Let $p$ be any rank one projection of $A$ such that $p\in\mathcal P$. Then $(p,0)$ is a rank one projection in $A_S$.  To prove the result it suffices to show that any rank one projection of $A_S$ belongs to the two-sided ideal generated by $(p,0)$ i.e. to $J_{(p,0)}$. First let $p^\prime$ be any other rank one projection of $A$. If $p^\prime\in E_p$, then there is a $w\in G_{\mathbf 1}(A)$ such that $p^\prime=wpw^{-1}$, and so $$(p^\prime,0)=(w,0)(p,0)(w^{-1},0)\in J_{(p,0)}.$$ If $p^\prime\notin E_p$, then $p^\prime=vqv^{-1}$ where $v\in G_{\mathbf 1}(A)$ and $q\in\mathcal P$ with $q\not=p$. 
   Observe now that $$(q,0)=(0,q\otimes p)(p,0)(0,p\otimes q)\in J_{(p,0)}, $$ whence 
   $$(p^\prime,0)=(v,0)(q,0)(v^{-1},0)\in J_{(p,0)}.$$ 
   Now let $(a,u)$ be an arbitrary rank one projection of $A_S$. Since $u$ is the limit of a sequence of elements in $A_{J,\pi}$, it follows by the preceding result that $(0,u)\in\cl(  \Soc(A_S))$. Hence, it follows that $(a,0)\in\cl(  \Soc(A_S))$. So the two-sided ideal generated by $a$, namely $J_a$, is inessential in $A$. It follows from Theorem~\ref{sgencomp} that for any member, say $b$, of $J_a$ the Riesz projections corresponding to the isolated nonzero spectral values of $b$ are finite rank elements of $A$.  Now take $x\in A$ such that $\sigma(ax)\not=\{0\}$. Then there is a finite rank projection, $p^\prime$, corresponding to $0\not=\lambda\in\sigma(ax)$ such that $\sigma(axp^\prime)=\{0,\lambda\}$ (or possibly $\{\lambda\}$ in the finite dimensional case). But we can write $p^\prime=\sum_{i=1}^kp_i$ where each $p_i$ is a rank one projection of $A$. If $p_iaxp_i=0$ for each $i$, then $\Tr(axp^\prime)=0$ which is absurd, since $\Tr(axp^\prime)$ is a positive integer multiple of $\lambda\not=0$. So we may assume without loss of generality that, say $p_1axp_1=\alpha p_1$, where $\alpha\not=0$ and moreover then that $\alpha=1$. Thus, if we take $(a,u)$ and then perform the operations (i) $(a,u)(x,0)=(ax,ux)$, (ii) $(p_1,0)(ax,ux)(p_1,0)=(p_1,0)$, we see that
$(p_1,0)\in J_{(a,u)}$. However, it was also shown in the first part of the proof that $(p_1,0)\in J_{(p,0)}$. So $J_{(a,u)}=J_{(p,0)}$, which shows that $\Soc\,A_S$ is a minimal two-sided ideal.
\end{proof}

\bibliographystyle{amsplain}
\bibliography{Spectral1}

\end{document}